\documentclass[11pt]{article}
\usepackage{graphicx}
 \usepackage{mathptmx,color}      
\usepackage[titletoc,title]{appendix}
\usepackage[numbers]{natbib}
\usepackage[T1]{fontenc}
\usepackage[applemac]{inputenc}
\usepackage[all]{xy}
\usepackage{enumerate}
\usepackage{comment}
\usepackage{amsmath,amssymb,stmaryrd}
\usepackage[british]{babel}
\usepackage{a4wide}
\usepackage[sans]{dsfont}
\usepackage{amsfonts}
\usepackage{amsthm}
\usepackage[mathscr]{euscript}
\usepackage{graphicx}
\usepackage{mathrsfs}
\usepackage{latexsym}
\usepackage{setspace}
\usepackage{amsthm,amsmath}

\newcommand{\p}{^}

\newcommand{\wt}{\widetilde}
\newcommand{\ol}{\overline}
	
	\newcommand{\lsi}{\llbracket}
		\newcommand{\rsi}{\rrbracket}
		
	\newcommand{\rmd}{\mathrm{d}}
\newcommand{\rme}{\mathrm{e}}

\def\cro#1{\langle #1\rangle}

\newcommand{\Lrm}{\mathrm{L}}

\newcommand{\Var}{\mathrm{Var}}

\newcommand{\Ascr}{\mathscr{A}}
\newcommand{\Bscr}{\mathscr{B}}

\newcommand{\Fscr}{\mathscr{F}}
\newcommand{\Gscr}{\mathscr{G}}
\newcommand{\Hscr}{\mathscr{H}}

\newcommand{\Jscr}{\mathscr{J}}

\newcommand{\Lscr}{\mathscr{L}}

\newcommand{\Oscr}{\mathscr{O}}
\newcommand{\Pscr}{\mathscr{P}}

\newcommand{\Rscr}{\mathscr{R}}
\newcommand{\Sscr}{\mathscr{S}}

\newcommand{\Xscr}{\mathscr{X}}
\newcommand{\cX}{\mathscr{X}}

\newcommand{\Zscr}{\mathscr{Z}}

\newcommand{\Mloc}{\Hscr\p1_{\mathrm{loc}}}

\newcommand{\et}{\eta}

\newcommand{\al}{\alpha}

\newcommand{\lm}{\lambda}

\newcommand{\sig}{\sigma}

\newcommand{\om}{\omega}
\newcommand{\Om}{\Omega}
\newcommand{\cadlag}{c\`adl\`ag\ }

\newcommand{\Abb}{\mathbb{A}}

\newcommand{\Ebb}{\mathbb{E}}

\newcommand{\Fbb}{\mathbb{F}}
\newcommand{\Xbb}{\mathbb{X}}
\newcommand{\Nbb}{\mathbb{N}}
\newcommand{\Pbb}{\mathbb{P}}

\newcommand{\Rbb}{\mathbb{R}}

\newcommand{\Hbb}{\mathbb{H}}
\newcommand{\Gbb}{\mathbb{G}}

\newcommand{\aPP}[2]{\ensuremath{\langle #1,#2 \rangle}}

\newtheorem{theorem}{Theorem}[section]
\newtheorem{lemma}[theorem]{Lemma}
\newtheorem{proposition}[theorem]{Proposition}
\newtheorem{corollary}[theorem]{Corollary}
\theoremstyle{definition}

\newtheorem{cexample}[theorem]{Counterexample}
\newtheorem{assumption}[theorem]{Assumption}

\newtheorem{remark}[theorem]{Remark}

\numberwithin{equation}{section}

\makeatletter
\def\timenow{\@tempcnta\time
\@tempcntb\@tempcnta
\divide\@tempcntb60
\ifnum10>\@tempcntb0\fi\number\@tempcntb
:\multiply\@tempcntb60
\advance\@tempcnta-\@tempcntb
\ifnum10>\@tempcnta0\fi\number\@tempcnta}
\makeatother

\title{Martingale Representation in the Enlargement \\of the Filtration Generated by a Point Process}
\author{Paolo Di Tella$^{1,2}$ and Monique Jeanblanc$^3$}
\date{}

\begin{document}
\maketitle
\begin{abstract}
Let $X$ be a point process and let $\Xbb$ denote the filtration generated by $X$. In this paper we study martingale representation theorems in the filtration $\Gbb$ obtained as an initial and progressive enlargement of the filtration $\Xbb$. The progressive enlargement is done here by means of a whole point process $H$. We do not require further assumptions on the point process $H$ nor on the dependence between $X$ and $H$. In particular, we recover the special case of the progressive enlargement by a random time $\tau$.
 \end{abstract}
{\noindent
\footnotetext[1]{Corresponding author.}
\footnotetext[2]{ \emph{Adress:} Inst.\ f\"ur Mathematische Stochastik, TU Dresden. Zellescher Weg 12-14 01069 Dresden, Germany.

\emph{E-Mail: }{\tt Paolo.Di\_Tella{\rm@}tu-dresden.de}}

\footnotetext[3]{ \emph{Adress:} Laboratoire de Math\'ematiques et Mod\'elisation d' \'Evry (LaMME), UMR CNRS 8071, Universit\'e d' \'Evry-Val-d'Essonne, Universit\'e Paris Saclay, 23 Boulevard de France, 91037 \'Evry cedex, France.

\emph{E-Mail: }{\tt monique.jeanblanc{\rm@}uni-evry.fr}}
}

\section{Introduction}
In this paper we study martingale representation theorems in the enlargement $\Gbb$ of the filtration $\Xbb$ generated by a point process $X$ (see Theorem \ref{thm:wrp.srp} below). The filtration $\Gbb$ here is obtained first enlarging $\Xbb$ initially by a sigma-field $\Rscr$ and then progressively by a whole point process $H$ (and not only by a random time). In other words, $\Gbb$ is the smallest right-continuous filtration containing $\Rscr$ and such that $X$ and $H$ are adapted. We show that in $\Gbb$ all local martingales can be represented as a stochastic integral with respect to a compensated random measure. Furthermore, we show a martingale representation theorem in $\Gbb$ with respect to \emph{three fundamental local martingales}. Due to the particular structure of the filtration $\Gbb$, we work without requiring any additional condition on the point processes $X$ and $H$. In this way, for the setting of point processes, we generalize all the results from literature (which we describe below) about the propagation of {martingale representation theorems} to an enlarged filtration. We notice that, thanks to the initial enlargement of $\Xbb$ by $\Rscr$, the important case of a point process $X$ with conditionally independent increments (Cox-point process) is also covered.

It is well-known that the multiplicity (see Davis and Varaiya \cite{dv:mi}) or the spanning number (see Duffie \cite{du:se}) of $\Xbb$ is equal to one. We shall show that the multiplicity of $\Gbb$ is, in general, less than or equal to three. However, we shall also discuss special cases in which the multiplicity of $\Gbb$ becomes less than or equal to two or it even reduces to one.

We now give an overview of the literature about the propagation of martingale representation theorems to the enlarged filtration. We denote by $\Fbb$ a reference filtration and by $\Gbb$ an enlargement of $\Fbb$. A local martingale $M$ is said to possess the predictable representation property (from now on the PRP) with respect to $\Fbb$ if all $\Fbb$-local martingales can be represented as a stochastic integral of predictable integrands  with respect to $M$.

In \cite{fo:sp}, Fontana proved that the PRP of a $d$-dimensional local martingale $M$ with respect to $\Fbb$ propagates to the \textit{initial} enlargement $\Gbb$ of $\Fbb$ by a random variable $L$ which satisfies \emph{Jacod's absolutely continuity hypothesis}. In this case, the multiplicity of $\Gbb$ remains less than or equal to $d$.

 A seminal paper about the propagation of the PRP to the \emph{progressive} enlargement $\Gbb$ by a random time $\tau$ is  Kusuoka \cite{Ku99}. In \cite{Ku99} the reference filtration $\Fbb$ is a Brownian filtration and $\tau$ is a random time satisfying, among other additional conditions, the \emph{immersion property} and \emph{avoiding} $\Fbb$-stopping times\footnote{A filtration $\Fbb\subseteq\Gbb$ is immersed in $\Gbb$ if  $\Fbb$-martingales remain $\Gbb$-martingales. Furthermore, we say that $\tau$ avoids $\Fbb$-stopping times if $\Pbb[\tau=\sig<+\infty]=0$ for all $\Fbb$-stopping time $\sig$.}. In the case of \cite{Ku99} the multiplicity of $\Gbb$ is equal to two. In \cite{js:tcr}, Jeanblanc and Song considered more general random times and they proved that the multiplicity of $\Gbb$ usually increases by one and for the case of honest times by two. In \cite{AJR18}, Aksamit, Jeanblanc and Rutkowski studied the propagation of the PRP to $\Gbb$ by a random time $\tau$ if $\Fbb$ is a Poisson filtration. The representation results obtained in \cite{AJR18} concern however only a special class of $\Gbb$-local martingales stopped at $\tau$.

Another important (and more general than the PRP) martingale representation theorem is the so called \emph{weak} representation property\footnote{The WRP is a representation of martingales as a sum of a stochastic integral of a predictable process with respect to a continuous local martingale plus a stochastic integral of a predictable function with respect to a compensated random measure, see \cite[Definition III.4.22]{JS00}.} (from now on WRP).

The propagation of the WRP to the progressive enlargement by a honest time $\tau$ has been established in \cite{ba:sf} by Barlow (see Theorem 6.1 therein).
More recently, under the immersion and the avoidance assumptions on $\tau$, the propagation of the WRP to $\Gbb$ has been established in Di Tella \cite{DT19} (see Theorem 5.3 therein).

If $\Fbb$ is progressively enlarged to $\Gbb$ by a \emph{whole process} instead of a random time, the propagation of martingale representation theorems to $\Gbb$ has been very little studied and all the existing results make use of independence assumptions.

For example, in \cite{X93}, Xue considered two semimartingales $X$ and $Y$ possessing the WRP each in its own filtration and, under the assumptions of the independence and the quasi-left continuity of $X$ and $Y$, the author shows that the WRP propagates to $\Gbb$. Similarly, in \cite{CT16}, Calzolari and Torti consider two local martingales $M$ and $N$ possessing the PRP each in its own filtration and then they study the propagation of the PRP to $\Gbb$. The assumptions in \cite{CT16} imply that $M$ and $N$ are independent (at least under an equivalent measure). In \cite{CT16} it is also shown that the multiplicity of $\Gbb$ is, in general, less than or equal to three. In  \cite{CT19}, the results of  \cite{CT16} are generalized to multidimensional martingales $M$ and $N$. A work-in-progress about the multiplicity of $\Gbb$ in this context is Calzolari and Torti \cite{CT17}. A more detailed comparison between  the existing literature and our results is postponed after the proof of Theorem  \ref{thm:wrp.srp} below.

The present paper has the following structure. In Section \ref{sec:bas}, we recall some basics on point processes, martingale theory and progressive enlargement of filtrations. In Section \ref{sec:PRP}, we establish the main results of the present paper, Theorem \ref{thm:wrp.srp} below, about the propagation of the WRP and the PRP. Section \ref{subs:orth} is devoted to the study of the orthogonality of the fundamental martingales which are used for the martingale representation in $\Gbb$. At the end of Section \ref{subs:orth} (see subsection \ref {subs:ex}) we study the propagation of the PRP in two examples, which, to the best of our knowledge, are new. In the Appendix we gather the proofs of some technical results.

\section{Basic Notions}\label{sec:bas}

Let $(\Om,\Fscr,\Pbb)$ be a probability space. We denote by $\Fbb=(\Fscr_t)_{t\geq0}$ a right-continuous filtration of subsets of $\Fscr$ and by $\Oscr(\Fbb)$ (resp.\ $\Pscr(\Fbb)$) the $\Fbb$-optional (resp.\ $\Fbb$-predictable) $\sig$-algebra on $\Om\times\Rbb_+$. For a \cadlag process $X$, we denote by $\Delta X$ the jump process and use the convention $\Delta X_0=0$.

For $p\geq1$, we denote by $\Hscr\p p(\Fbb)$ the space of $\Fbb$-adapted $p$-integrable martingales (see \cite[Chapter II, Section 1]{J79}). The space $\Hscr\p p_\mathrm{loc}(\Fbb)$ is introduced from $\Hscr\p p(\Fbb)$ by localization. The space $\Hscr\p1_\mathrm{loc}(\Fbb)$ coincides with the set  of \emph{all} $\Fbb$-local martingales (see \cite[Lemma 2.38]{J79}). We set $\Hscr\p p_0(\Fbb):=\{X\in\Hscr\p p(\Fbb): X_0=0\}$. The space  $\Hscr\p p_{\mathrm{loc},0}(\Fbb)$ is defined analogously.

We denote by $\Ascr\p+=\Ascr\p+(\Fbb)$ the space of $\Fbb$-adapted integrable increasing processes (see \cite[Definition I.3.1 and I.3.6]{JS00}) and by $\Ascr\p+_\mathrm{loc}=\Ascr\p+_\mathrm{loc}(\Fbb)$ the localized version of $\Ascr\p+$. By $\Ascr=\Ascr(\Fbb)$ (resp., $\Ascr_\mathrm{loc}=\Ascr_\mathrm{loc}(\Fbb)$)
we denote the space of  $\Fbb$-adapted processes of integrable (resp., locally integrable) variation.

For $X\in\Ascr_\mathrm{loc}$ and for a  nonnegative measurable process $K$, we indicate
 by $K\cdot X=(K\cdot X_t)_{t\geq0}$ the process defined by the (Lebesgue--Stieltjes) integral of $K$ with respect to $X$, that is, $K\cdot X_t(\om):=\int_0\p t K_s(\om)\rmd X_s(\om)$, if $\int_0\p t K_s(\om)\rmd \Var(X)_s(\om)<+\infty$, for every $\om\in\Om$ and $t\geq0$, where $t\mapsto\Var(X)_t(\om)$ denotes the total variation of $s\mapsto X_s(\om)$ on $[0,t]$, $t\geq0$.

\textbf{Point processes.} A \emph{point process} $X$ with respect to $\Fbb$ is an $\Nbb$-valued and $\Fbb$-adapted  increasing process such that $\Delta X\in\{0,1\}$. A point process $X$ is locally bounded and therefore $X\in\Ascr\p+_\mathrm{loc}$. We denote by $X\p{p,\Fbb}\in\Ascr\p+_\mathrm{loc}$ the $\Fbb$-dual predictable projection of $X$ (see \cite[Theorem I.3.17]{JS00}) and define $\overline X\p\Fbb:=X-X\p{p,\Fbb}\in\Mloc(\Fbb)\cap\Ascr_\mathrm{loc}(\Fbb)$. We recall that $\Delta X\p{p,\Fbb}\leq 1$ (see \cite[Corollary 5.28]{HWY92}).

For a point process $X$, $\Lrm\p p(\overline X\p\Fbb)=\Lrm\p p(\overline X\p\Fbb,\Fbb)$ indicates the space of $\Fbb$-predictable processes $K$ such that $(K\p2\cdot[\overline X\p\Fbb,\overline X\p\Fbb])\p{p/2}\in\Ascr\p+(\Fbb)$, where $[\cdot,\cdot]$ denotes the quadratic variation process. The space $\Lrm\p p_\mathrm{loc}(\overline X\p\Fbb)=\Lrm\p p_\mathrm{loc}(\overline X\p\Fbb,\Fbb)$ is defined analogously but making use of $\Ascr\p+_\mathrm{loc}(\Fbb)$ instead. For $K\in\Lrm_\mathrm{loc}\p1(\overline X\p\Fbb)$, $K\cdot \overline X\p\Fbb$ indicates the stochastic integral of $K$ with respect to $\overline X\p\Fbb$: It is a local martingale starting at zero and, by \cite[Proposition 2.46 (b)]{J79}, $K\cdot \overline X\p\Fbb\in\Hscr\p p(\Fbb)$ if and only if $K\in\Lrm\p p(\overline X\p\Fbb)$. We observe that if $K\in \Lrm_\mathrm{loc}\p1(\overline X\p\Fbb)$, then the stochastic integral $K\cdot \overline X\p\Fbb$ coincides with the Lebesgue--Stieltjes integral, whenever this latter one exists and is finite.

\textbf{Random measures.} For a Borel subset $E$ of $\Rbb\p d$, we introduce $\wt\Om:=\Om\times\Rbb_+\times E$ and the product $\sig$-algebras $\wt\Oscr(\Fbb):=\Oscr(\Fbb)\otimes\Bscr(E)$ and $\wt\Pscr(\Fbb):=\Pscr(\Fbb)\otimes\Bscr(E)$. If $W$ is an $\wt\Oscr(\Fbb)$-measurable (resp.\ $\wt\Pscr(\Fbb)$-measurable) mapping from $\wt\Om$ into $\Rbb$, it is called an $\Fbb$-optional (resp.\ $\Fbb$-predictable) function.

Let $\mu$ be a random measure  on $\Rbb_+\times E$ (see \cite[Definition II.1.3]{JS00}).
For a nonnegative $\Fbb$-optional function $W$, we write $W\ast\mu=(W\ast\mu_t)_{t\geq0}$, where $W\ast\mu_t(\om):=\int_{[0,t]\times E}W(\om,s,x)\mu(\om,\rmd s,\rmd x)$ is the process defined by the (Lebesgue--Stieltjes) integral of $W$ with respect to $\mu$ (see \cite[II.1.5]{JS00} for details). If $W\ast\mu$ is $\Fbb$-optional (resp.\ $\Fbb$-predictable), for every optional (resp.\ $\Fbb$-predictable) function $W$, then $\mu$ is called $\Fbb$-optional (resp.\ $\Fbb$-predictable).

\textbf{Semimartingales.}
For an $\Rbb\p d$-valued $\Fbb$-semimartingale $X$, we denote by $\mu\p X$ the jump measure of $X$, that is,
$
\mu\p{X}(\om,\rmd t,\rmd x)=\sum_{s>0}1_{\{\Delta X_s(\om)\neq0\}}\delta_{(s,\Delta X_s(\om))}(\rmd t,\rmd x)$, where, here and in the whole paper, $\delta_a$ denotes the Dirac measure at point $a$.
From \cite[Theorem II.1.16]{JS00}, $\mu\p X$ is an \emph{integer-valued random measure} on $\Rbb_+\times \Rbb\p d$ with respect to $\Fbb$ (see \cite[Definition II.1.13]{JS00}). Thus, $\mu\p X$ is, in particular, an $\Fbb$-optional random measure. According to \cite[Definition III.1.23]{JS00}, $\mu\p X$ is called an $\Rbb\p d$-valued \emph{marked point process} (with respect to $\Fbb$) if $\mu\p X(\om;[0,t]\times \Rbb\p d)<+\infty$, for every $\om\in\Om$ and $t\in\Rbb_+$. By $\nu\p X$ we denote the $\Fbb$-predictable compensator of $\mu\p X$ (see \cite[Definition II.2.6]{JS00}) and $\nu\p X$ is a predictable random measure characterized by the following properties: For any $\Fbb$-predictable function $W$ such that $|W|\ast\mu\p X\in\Ascr\p +_\mathrm{loc}$, we have $|W|\ast\nu\p X\in\Ascr\p +_\mathrm{loc}$ and $W\ast\mu\p X-W\ast\nu\p X\in\Mloc(\Fbb)$.

\section{Martingale Representation}\label{sec:PRP}
For a point process $X$, we introduce the filtration $\Xbb=(\Xscr_t)_{t\geq0}$ by $\Xscr_t:=\bigcap_{s>t}\sig(X_u,\ 0\leq u\leq s)$, $t\geq0$. Let $\Rscr$ be a $\sig$-field, called the \emph{initial} $\sig$-\emph{field}. The filtration $\Fbb=(\Fscr_t)_{t\geq0}$ is defined by $\Fscr_t:=\Rscr\vee\Xscr_t$. Thus, $\Fbb$ is the \emph{initial enlargement} of $\Xbb$ by the $\sig$-field $\Rscr$. Clearly, $X$ is a point process with respect to $\Fbb$. However, the choice of the $\sig$-field $\Rscr$ affects the predictable compensator of $X$: In general, $X\p{p,\Xbb}$ does not coincide with $X\p{p,\Fbb}$. A non-trivial $\Rscr$ allows to include in the theory developed in the present paper, without any additional effort, also the important case of a Cox-point process $X$ with respect to $\Fbb$, i.e., $X\p{p,\Fbb}$ is $\Fscr_0$-measurable or, equivalently,  $X$ has conditionally independent increments with respect to $\Fbb$ given $\Fscr_0$.

\begin{lemma}\label{lem:prp.pp}
\textnormal{(i)} The filtration $\Fbb$ is right-continuous.

\textnormal{(ii)} For every $Y\in\Mloc(\Fbb)$ there exists $K\in\Lrm\p1_\mathrm{loc}(\overline X\p{\Fbb})$ such that $Y=Y_0+K\cdot \overline X\p{\Fbb}$, that is, the $\Fbb$-local martingale $\overline X\p{\Fbb}:=X-X\p{p,\Fbb}$ possesses the \emph{PRP} with respect to $\Fbb$.
\end{lemma}
\begin{proof} Point (i) follows by \cite[Proposition 3.39 (a)]{J79}, while (ii) is \cite[Eq.\ (III.4.38)]{JS00}.\end{proof}

We consider a point process $H$ and introduce $\Hbb=(\Hscr_t)_{t\geq0}$ by $\Hscr_t:=\bigcap_{s>t}\sig(H_u,\ 0\leq u\leq s)$, $t\geq0$. We denote by $\Gbb=(\Gscr_t)_{t\geq0}$ the \emph{progressive enlargement} of $\Fbb$ by $\Hbb$: That is, $\Gbb$ is defined by
\[
 \Gscr_t:=\bigcap_{s>t} \Fscr_{s}\vee\Hscr_{s} \quad t\geq0.
\]
In other words, $\Gbb$ is the \emph{smallest} right-continuous filtration containing $\Fbb$ and $\Hbb$.

 As a special example of $H$, one can take the default process associated with a random time $\tau$, i.e., $H_t(\om):=1_{\lsi \tau,+\infty\lsi}(\om,t)$, where $\tau$ is a $(0,+\infty]$-valued random variable. In this case, $\Gbb$ is called the progressive enlargement of $\Fbb$ by $\tau$ and it is the smallest right-continuous filtration containing $\Fbb$ and such that $\tau$ is a $\Gbb$-stopping time.

We are now in the following situation: The local martingale $\overline X\p\Fbb$ has the PRP with respect to $\Fbb$ and the local martingale $\overline H\p\Hbb$ has the PRP with respect to $\Hbb$. Our aim is to investigate how to represent martingales of the filtration $\Gbb$.

 The proof of the next proposition is postponed to the appendix.
\begin{proposition}\label{prop:poi.pr.pb}
Let $X$ and $H$ be two point processes with respect to $\Fbb$ and $\Hbb$, respectively. Then, the processes $X-[X,H]$, $H-[X,H]$ and $[X,H]$ are point processes with respect to $\Gbb$. Furthermore, they have pairwise no common jumps.
\end{proposition}

We introduce the $\Rbb\p2$-valued $\Gbb$-semimartingale $\wt X=(X,H)\p\top$. The jump measure $\mu\p{\wt X}$ of $\wt X$ is an integer-valued random measure on $\Rbb_+\times E$, where $E:=\{(1,0);(0,1);(1,1)\}$. For a $\Gbb$-predictable function $W$ on $\wt\Om:=\Om\times\Rbb_+\times E$ we define the $\Gbb$-predictable processes
$W(1,0)$, $W(0,1)$ and $W(1,1)$ by $W_t(0,1):=W(t,0,1)$, $W_t(1,0):=W(t,1,0)$,  $W_t(1,1):=W(t,1,1)$, $t\geq0$.

\begin{proposition}\label{prop:mar.poi.pr} Let us consider the $\Rbb\p2$-valued $\Gbb$-semimartingale $\wt X=(X,H)\p\top$. We then have:

(i) The jump measure $\mu\p{\wt X}$ of $\wt X$ on $\Rbb_+\times E$, where $E:=\{(1,0);(0,1);(1,1)\}$, is given by
\begin{equation}\label{eq:ju.mea.spp}
\begin{split}
\mu\p{\wt X}(\rmd t,\rmd x_1,\rmd x_2)=\rmd (X_t-&[X,H]_t)\delta_{(1,0)}(\rmd x_1,\rmd x_2)\\&+\rmd (H_t-[X,H]_t)\delta_{(0,1)}(\rmd x_1,\rmd x_2)+\rmd [X,H]_t\delta_{(1,1)}(\rmd x_1,\rmd x_2).
\end{split}
\end{equation}
Furthermore, $\mu\p{\wt X}$ is an $\Rbb\p2$-valued marked point process with respect to $\Gbb$.

(ii) The $\Gbb$-predictable compensator $\nu\p{\wt X}$ of $\mu\p{\wt X}$ is given by
\begin{equation}\label{eq:ju.mea.spp.com}
\begin{split}
\nu\p{\wt X}(\rmd t,\rmd x_1,\rmd x_2)=\rmd(X_t-&[X,H]_t)\p{p,\Gbb}\delta_{(1,0)}(\rmd x_1,\rmd x_2)\\&+\rmd(H_t-[X,H]_t)\p{p,\Gbb}\delta_{(0,1)}(\rmd x_1,\rmd x_2)+\rmd [X,H]_t\p{p,\Gbb}\delta_{(1,1)}(\rmd x_1,\rmd x_2).
\end{split}
\end{equation}
\end{proposition}
\begin{proof} We start verifying (i).
By the definition of the jump measure of $\wt X$, we get:
\[
\begin{split}
\mu\p{\wt X}(\om,\rmd t,\rmd x_1,\rmd x_2)&= \sum_{s>0}\Delta X_s(\om)(1- \Delta H_s(\om))\delta_{(s,1,0)}(\rmd t,\rmd x_1,\rmd x_2)
\\&+\sum_{s>0}\Delta H_s(\om)(1- \Delta X_s(\om))\delta_{(s,0,1)}(\rmd t,\rmd x_1,\rmd x_2)
\\&+\sum_{s>0}\Delta X_s(\om)\Delta H_s(\om)\delta_{(s,1,1)}(\rmd t,\rmd x_1,\rmd x_2),
\end{split}
\]
which is \eqref{eq:ju.mea.spp}. That is,
$
\mu\p{\wt X}(\om,[0,t]\times\Rbb\p2)=X_t(\om)-[X,H]_t(\om)+H_t(\om)<+\infty$,
$\om\in\Om$, $t\geq0$. Thus, $\mu\p{\wt X}$ is an $\Rbb\p 2$-valued marked point process with respect to $\Gbb$. We now come to (ii). Let us denote by $\nu$ the $\Gbb$-predictable random measure on the right-hand side of \eqref{eq:ju.mea.spp.com}. We have to show that $\nu$ coincides with $\nu\p{\wt X}$. So, let $W$ be a $\Gbb$-predictable function such that $|W|\ast\mu\p{\wt X}\in\Ascr\p+_\mathrm{loc}(\Gbb)$. By (i) we get
\begin{equation}\label{eq:Wastmu}
|W|\ast\mu\p{\wt X}=|W(1,0)|\cdot(X-[X,H])+|W(0,1)|\cdot(H-[X,H])+|W(1,1)|\cdot[X,H].
\end{equation}
We now denote by $Y\p 1$, $Y\p2$ and $Y\p 3$ the first, the second and the third term on the right-hand side of \eqref{eq:Wastmu}, respectively. By Proposition \ref{prop:poi.pr.pb}, $X-[X,H]$, $H-[X,H]$ and $[X,H]$ being point processes, we deduce from \eqref{eq:Wastmu} that $Y\p i\in\Ascr\p+_\mathrm{loc}(\Gbb)$, $i=1,2,3$, since $|W|\ast\mu\p{\wt X}\in\Ascr\p+_\mathrm{loc}(\Gbb)$. Then, the $\Gbb$-dual predictable projection $(Y\p i)\p{p,\Gbb}$ of $Y\p i$, $i=1,2,3$, exists and it belongs to $\Ascr\p+_\mathrm{loc}(\Gbb)$. Since $W(1,0)$, $W(0,1)$ and $W(1,1)$ are $\Gbb$-predictable processes, the properties of the $\Gbb$-dual predictable projection yield the identities $
(Y\p 1)\p{p,\Gbb}=|W(1,0)|\cdot(X-[X,H])\p{p,\Gbb}$, $(Y\p 2)\p{p,\Gbb}=|W(0,1)|\cdot(H-[X,H])\p{p,\Gbb}$ and finally $(Y\p 3)\p{p,\Gbb}=|W(1,1)|\cdot[X,H]\p{p,\Gbb}$. Thus, $|W|\ast\nu=(Y\p 1)\p{p,\Gbb}+(Y\p 2)\p{p,\Gbb}+(Y\p 3)\p{p,\Gbb}$ because of the definition of $\nu$. So, $|W|\ast\nu\in\Ascr\p+_\mathrm{loc}(\Gbb)$ holds, since $(Y\p i)\p{p,\Gbb}\in\Ascr\p+_\mathrm{loc}(\Gbb)$, $i=1,2,3$. We now show that $W\ast\mu\p{\wt X}-W\ast\nu\in\Mloc(\Gbb)$: The linearity of the integral with respect to the \emph{integrator} and (i)  yield
 \[
\begin{split}
W\ast\mu\p{\wt X}-W\ast\nu&=W(1,0)\cdot((X-[X,H])-(X-[X,H])\p{p,\Gbb})\\&+W(0,1)\cdot((H-[X,H])-(H-[X,H])\p{p,\Gbb})\\&+W(1,1)\cdot([X,H]-[X,H]\p{p,\Gbb})\in\Mloc(\Gbb).
\end{split}
\]
Hence, $\nu\p{\wt X}=\nu$ and the proof is complete.
\end{proof}
We now denote by $\wt\Gbb=(\wt\Gscr_t)_{t\geq0}$ the smallest right-continuous filtration such that $\mu\p{\wt X}$ is optional and introduce $\Gbb\p\ast=(\Gscr\p\ast_t)_{t\geq0}$ by $\Gscr\p\ast_t:=\Rscr\vee\wt\Gscr_t$.  From \cite[Proposition 3.39 (a)]{J79},  $\Gbb\p\ast$ is \emph{right-continuous} and, hence, it is the smallest right-continuous filtration satisfying $\Rscr\subseteq\Gscr_0$ and such that $\mu\p{\wt X}$ is an optional random measure.
The proof of the following lemma is postponed to the appendix.
\begin{lemma}\label{lem:fil.eq} The filtration $\Gbb\p\ast$ coincides with $\Gbb$.
\end{lemma}

As announced in the introduction, we give two kinds of representations of martingales: The first one with respect to a compensated random measure, the second one with respect to the following three locally bounded $\Gbb$-local martingales:
\begin{equation}\label{eq:mart.Z1.Z2.Z3}
Z\p1:=\ol{X-[X,H]}\p{\,\Gbb},\quad Z\p2:=\ol{H-[X,H]}\p{\,\Gbb},\quad Z\p3:=\ol{[X,H]}\p{\,\Gbb},
\end{equation}
where, we recall, for a point process $Z$, we define $\ol Z\p\Gbb:=Z-Z\p{p,\Gbb}$.
\begin{theorem}\label{thm:wrp.srp}
\textnormal{(i)} Let $Y\in\Mloc(\Gbb)$. Then, there exists a $\Gbb$-predictable function $W$ such that $|W|\ast\mu\p{\wt X}$ belongs to $\Ascr\p+_\mathrm{loc}(\Gbb)$ and
\begin{equation}\label{eq:wrp} Y=Y_0+W\ast\mu\p{\wt X}-W\ast\nu\p{\wt X}.
\end{equation}
\textnormal{(ii)}  Let $Y\in\Mloc(\Gbb)$. Then $Y$ has the following representation:
\begin{equation}\label{eq:prp.spp}
Y=Y_0+K\p1\cdot Z\p1+K\p2\cdot Z\p2+K\p3\cdot Z\p3,\quad K\p i\in\Lrm\p1_\mathrm{loc}(Z\p i,\Gbb),\quad i=1,2,3.
\end{equation}
\end{theorem}
\begin{proof}
We first verify (i). Because of Proposition \ref{prop:mar.poi.pr} (i), the random measure $\mu\p{\wt X}$ is a marked point process with respect to $\Gbb$. Furthermore, by Lemma \ref{lem:fil.eq}, the filtration $\Gbb$ coincides with $\Gbb\p\ast$. Hence,  \eqref{eq:wrp} follows by \cite[Theorem III.4.37]{JS00}.

We now come to (ii). Let $Y\in\Mloc(\Gbb)$. By (i), there exists a $\Gbb$-predictable function $W$ such that $|W|\ast\mu\p{\wt X}\in\Ascr\p+_\mathrm{loc}(\Gbb)$ and $Y=Y_0+W\ast\mu\p{\wt X}-W\ast\nu\p{\wt X}$. Since $|W|\ast\mu\p{\wt X}\in\Ascr\p+_\mathrm{loc}(\Gbb)$, the $\Gbb$-predictable functions $W\p{j,k}(\om,t,x_1,x_2):=W(\om,t,x_1,x_2)1_{\{x_1=j,x_2=k\}}$, $(j,k)\in E=\{(1,0);(0,1);(1,1)\}$, satisfy $|W\p{j,k}|\ast\mu\p{\wt X}\in\Ascr\p+_\mathrm{loc}(\Gbb)$. Therefore, we obtain $\big(W\p{j,k}\ast\mu\p{\wt X}-W\p{j,k}\ast\nu\p{\wt X}\big)\in\Mloc(\Gbb)$. From \eqref{eq:ju.mea.spp} and \eqref{eq:ju.mea.spp.com}, the relations
\begin{gather}
W\p{1,0}\ast\mu\p{\wt X}-W\p{1,0}\ast\nu\p{\wt X}=W(1,0)\cdot Z\p1\label{int1},\\  W\p{0,1}\ast\mu\p{\wt X}-W\p{0,1}\ast\nu\p{\wt X}=W(0,1)\cdot Z\p2,\label{int2}\\  W\p{1,1}\ast\mu\p{\wt X}-W\p{1,1}\ast\nu\p{\wt X}=W(1,1)\cdot Z\p3\label{int3}
\end{gather}
hold, where the integrals on the right-hand side in these identities are Lebesgue--Stieltjes integrals. So from \eqref{eq:wrp}, Proposition \ref{prop:mar.poi.pr} and \eqref{int1}, \eqref{int2}, \eqref{int3}, we get
\[\begin{split}
Y&=Y_0+ W\ast\mu\p{\wt X}-W\ast\nu\p{\wt X}=Y_0+W(1,0)\cdot Z\p1+W(0,1)\cdot Z\p2+W(1,1)\cdot Z\p3,
\end{split}\]
the latter term being a sum of Lebesgue--Stieltjes integrals. For shortness, we denote $K\p1:=W(1,0)$, $K\p2:=W(0,1)$ and $K\p3:=W(1,1)$.  It remains to show that $K\p i\in\Lrm\p1_\mathrm{loc}(Z\p i,\Gbb)$, $i=1,2,3$. By Proposition \ref{prop:mar.poi.pr} and \eqref{eq:mart.Z1.Z2.Z3}, the estimate
$\sum_{0\leq s\leq\cdot}|K\p 1_s\Delta Z\p 1_s|\leq|W\p{1,0}|\ast\mu\p{\wt X}+|W\p{1,0}|\ast\nu\p{\wt X}\in\Ascr\p+_\mathrm{loc}(\Gbb)$ holds. Similarly, we get  $\sum_{0\leq s\leq\cdot}|K\p 2_s\Delta Z\p 2_s|,\ \sum_{0\leq s\leq\cdot}|K\p 3_s\Delta Z\p 3_s|\in\Ascr\p+_\mathrm{loc}(\Gbb)$. Hence, \cite[Theorem 9.5 (1)]{HWY92}, yields $K\p i\in\Lrm\p1_\mathrm{loc}(Z\p i,\Gbb)$, $i=1,2,3$.  The proof of the theorem is complete.
\end{proof}
Theorem \ref{thm:wrp.srp} shows that, without requiring any further condition on the point processes $X$ and $H$,  the martingale representation property of the process $\ol X\p\Fbb$ in $\Fbb$ and of $\ol H\p\Hbb$ in $\Hbb$ propagates to $\Gbb$ both with respect to the compensated random measure $\mu\p{\wt X}-\nu\p{\wt X}$ and with respect to the three local martingales $Z\p1$, $Z\p2$ and $Z\p3$.

\paragraph*{Comparison with the existing literature.} We now review some results from the literature which are generalized by Theorem \ref{thm:wrp.srp}, at least in the special case of point processes.

Before, we recall that two local martingales  $M$ and $N$ with respect to a joint filtration, say $\Abb$, are called \emph{orthogonal} if $MN\in\Hscr\p1_{\mathrm{loc}}(\Abb)$ (or, equivalently, if $[M,N]\in\Hscr\p1_{\mathrm{loc}}(\Abb)$). If $M,N\in\Hscr\p2_{\mathrm{loc}}(\Abb)$, then they are orthogonal if and only if $\aPP{M}{N}=0$, where $\aPP{M}{N}$ denotes the predictable co-variation with respect to $\Abb$.
\begin{itemize}
\item For a process $Y$ we denote by $\Fbb\p Y$ the smallest right-continuous filtration such that $Y$ is adapted. In Calzolari and Torti \cite{CT16}, it is assumed that $M\in\Hscr\p2(\Fbb\p M)$ and $N\in\Hscr\p2(\Fbb\p N)$ possess the PRP with respect to $\Fbb\p M$ and $\Fbb\p N$, respectively and that $\Fscr\p M_0$ and $\Fscr\p N_0$ are trivial. Moreover, the authors assume the following conditions: (1) $M,N\in\Hscr\p2(\Gbb)$, $\Gbb:=\Fbb\p M\vee\Fbb\p M$, and (2) $M$ and $N$ are orthogonal in $\Gbb$, that is, $[M,N]\in\Mloc(\Gbb)$. Under these assumptions it is shown in \cite[Theorem 4.5]{CT16} that $\Fbb\p M$ and $\Fbb\p N$ are independent; that $M,N,[M,N]\in\Hscr\p2(\Gbb)$ are pairwise orthogonal and that every $Y\in\Hscr\p2(\Gbb)$ can be represented as
\[
Y=Y_0+K\p1\cdot M+K\p2\cdot N+K\p3\cdot [M,N], \quad K\p1\in\Lrm\p2(M,\Gbb),\  K\p2\in\Lrm\p2(N,\Gbb),\  K\p3\in\Lrm\p2([M,N],\Gbb).
\]
In the setting of point processes, Theorem \ref{thm:wrp.srp} generalizes \cite[Theorem 4.5]{CT16} since we do not make the very strong assumptions that $\ol X\p{\Fbb}$ and $\ol H\p{\Hbb}$ remain $\Gbb$-local martingales and that they are orthogonal in $\Gbb$. Moreover, we do not require the triviality of $\Fscr_0$ and, nevertheless, we show that the PRP propagates to $\Gbb$.

\item In  Di Tella \cite[Theorem 5.3]{DT19}  an  $\Rbb\p d$-dimensional semimartingale $X$ possessing the WRP with respect to a filtration $\Fbb$ is considered. The filtration $\Fbb$ is enlarged progressively to $\Gbb$ by a random time $\tau$ avoiding $\Fbb$ stopping times and such that $\Fbb$ is immersed in $\Gbb$. Then it is shown that the semimartingale $(X,H)$ possesses the WRP with respect to $\Gbb$, where $H=1_{\lsi\tau,+\infty\lsi}$ (see \cite[Definition 3.1]{DT19}). Theorem \ref{thm:wrp.srp} above generalizes \cite[Theorem 5.3]{DT19} because, in the special case of a point process $X$ and $H=1_{\lsi\tau,+\infty\lsi}$, it shows that the weak representation property propagates to $\Gbb$, without any further assumption on $\tau$.

\item In \cite{AJR18}, Aksamit, Jeanblanc and Rutkowski considered the case of a Poisson process $X$ with respect to $\Xbb$ and enlarged $\Xbb$ progressively by a random time $\tau$. They obtained a predictable representation result for a particular class of $\Gbb$-martingales stopped at $\tau$. Theorem \ref{thm:wrp.srp} generalizes the results of \cite{AJR18} because  we obtain a predictable representation result which is valid for every $\Gbb$-local martingale and not only for a special class of $\Gbb$-local martingales stopped at $\tau$: In other words, our martingale representation theorem goes \textit{beyond} $\tau$. Therefore, we get as a by-product a martingale representation theorem for \emph{all} $\Gbb$-martingales stopped at $\tau$ (see \S\ref{subs:enl.rt}) and not only for a particular class. We also stress that the proofs exhibited in the present paper are more elementary and direct than those given in \cite{AJR18}.

\item In \cite{X93}, Xue has shown that the weak representation property of a \emph{quasi-left continuous}\footnote{ An $\Fbb$-adapted \cadlag process $X$ is called \emph{quasi-left continuous} if $\Delta X_\sig=0$ for every finite-valued and $\Fbb$-predictable stopping time $\sig$.
} semimartingale $X$ with respect to a filtration $\Fbb$ propagates to the filtration $\Gbb=\Fbb\vee\Hbb$, where $\Hbb$ is a filtration independent of $\Fbb$ and furthermore $\Hbb$ supports a quasi-left continuous semimartingale $Y$ possessing the weak representation property with respect to $\Hbb$. If $X$ and $H$ are point processes, Theorem \ref{thm:wrp.srp} generalizes the work of Xue because we require neither the independence of $\Fbb$ and $\Hbb$ nor the quasi-left continuity of $X$ and $H$.
\end{itemize}

In the remaining part of this section we are going to study the \emph{stable subspaces} generated by the local martingales $Z\p1$, $Z\p2$ and $Z\p3$ in $\Hscr\p p_0(\Gbb)$, for $p\geq1$. We now shortly recall the definition of the stable subspace \emph{generated} by a family of local martingales of locally integrable variation, which is sufficient in this context. The main reference about stable subspaces is \cite[Chapter IV]{J79}.

 For $Y\in\Hscr\p p(\Gbb)$, $p\geq1$, we set $\|Y\|_{\Hscr\p p}:=\Ebb[\sup_{t\geq0}|Y_t|\p p]\p{1/p}$. For $p=2$, we also introduce the equivalent norm $\|Y\|_2:=\Ebb[Y\p2_\infty]\p{1/2}$ and $(\Hscr\p 2,\|\cdot\|_{2})$ is a Hilbert space.

Let $\Hscr$ be a linear subspace of $(\Hscr\p p_0(\Gbb), \|\cdot\|_{\Hscr\p p})$. Then $\Hscr$ is called \emph{a stable subspace} of $\Hscr\p p_0(\Gbb)$ if $\Hscr$ is closed in $(\Hscr\p p_0(\Gbb), \|\cdot\|_{\Hscr\p p})$ and \emph{stable under stopping}, that is, for every $\Gbb$-stopping time $\sig$, if $Z\in\Hscr$, then $Z\p\sig\in\Hscr$. For a family $\Zscr\subseteq\Hscr\p p_{\mathrm{loc},0}(\Gbb)\cap\Ascr_\mathrm{loc}(\Gbb)$, the stable subspace generated by $\Zscr$ in $\Hscr\p p_0(\Gbb)$ is denoted by $\Lscr\p p_\Gbb(\Zscr)$ and it is defined as the closure in $(\Hscr\p p_0(\Gbb),\|\cdot\|_{\Hscr\p p})$ of the set $\{K\cdot Z,\ Z\in\Zscr,\ K\in\Lrm\p p(Z,\Gbb)\}$. Notice that $\Lscr\p p_\Gbb(\Zscr)$ is characterized as following: It is the smallest stable subspace of $\Hscr\p p_0(\Gbb)$ containing the set $\Sscr:=\{Z\p\sig,\ \sig\ \Gbb\textnormal{-stopping time}: Z\p\sig\in\Hscr\p p(\Gbb)\}$ (see also \cite[Definition 4.4]{J79} and the subsequent comment therein).

\begin{corollary}[to Theorem \ref{thm:wrp.srp}]\label{cor:st.ssp.Z}
Let $p\geq1$ and let $\Zscr:=\{Z\p1,Z\p2,Z\p3\}$.

\textnormal{(i)} The identity $\Lscr\p p_\Gbb(\Zscr)=\Hscr\p p_0(\Gbb)$ holds.

\textnormal{(ii)} If furthermore $p=2$ and the local martingales $Z\p1, Z\p2, Z\p3\in\Hscr\p2_{\mathrm{loc},0}(\Gbb)$ are pairwise orthogonal, then every $Y\in\Hscr\p2(\Gbb)$ can be represented as
\[
Y=Y_0+K\p1\cdot Z\p1+K\p2\cdot Z\p2+K\p3\cdot Z\p3,\quad K\p i\in\Lrm\p2(Z\p i,\Gbb),\quad i=1,2,3,
\]
and this is an orthogonal decomposition of $Y$ in $(\Hscr\p2(\Gbb), \|\cdot\|_{2})$.
\end{corollary}
\begin{proof}
We first verify (i). Let $Y$ be a \emph{bounded} $\Gbb$-martingale orthogonal to $\Zscr$ and satisfying $Y_0=0$. We are going to show that $Y=0$ identically. By the orthogonality of $Y$ and $\Zscr$, we have $[Y,Z\p i]\in\Hscr\p 1_{\mathrm{loc},0}(\Gbb)$, for $i=1,2,3$. By Theorem \ref{thm:wrp.srp} (ii), we can represent $Y$ as in \eqref{eq:prp.spp}. Hence,
\begin{equation}\label{eq:pn.Y}
[Y,Y]=K\p1\cdot[Y,Z\p1]+K\p2\cdot[Y,Z\p2]+K\p3\cdot[Y,Z\p3],\quad K\p i\in\Lrm\p1_\mathrm{loc}(Z\p i,\Gbb).
\end{equation}
Now, using that $Y$ is bounded by a constant, say $C>0$, and that $K\p i\in\Lrm\p1_\mathrm{loc}(Z\p i,\Gbb)$, we obtain
\[
\big((K\p i)\p2\cdot[[Y,Z\p i],[Y,Z\p i]]\big)\p{1/2}=\big(\sum(K\p i\Delta Y \Delta Z\p i)\p2\big)\p{1/2}\leq 2\,C\big((K\p i)\p2\cdot[Z\p i,Z\p i]\big)\p{1/2}\in\Ascr\p+_\mathrm{loc}(\Gbb),\quad i=1,2,3.
\]

This shows that each addend on the right-hand side of \eqref{eq:pn.Y} belongs to $\Hscr\p 1_{\mathrm{loc},0}(\Gbb)$. Therefore, we also have $[Y,Y]\in\Hscr\p 1_{\mathrm{loc},0}(\Gbb)$, implying that $Y$ is orthogonal to itself. Hence, by \cite[Lemma I.4.13 (a)]{JS00}, we get that $Y=0$ identically. This implies $\Lscr\p1_\Gbb(\Zscr)=\Hscr_0\p1(\Gbb)$ by \cite[Corollary 4.12]{J79}. If $Y\in\Hscr\p 1_{\mathrm{loc},0}(\Gbb)$ is orthogonal to $\Zscr$ (and not necessarily bounded), then,  applying \cite[Proposition 4.67]{J79}, we also get that $Y=0$ identically. Hence, \cite[Proposition 4.10 (b)]{J79} yields that if $Y\in\Hscr\p p_{\mathrm{loc},0}(\Gbb)$ is orthogonal to $\Zscr$, then $Y=0$ identically, for every $p\in[1,+\infty]$. Thus, since for every $p\geq1$ we have $\Zscr\subseteq\Hscr\p p_{\mathrm{loc},0}(\Gbb)$, \cite[Theorem 4.11 (b)]{J79} yields $\Lscr\p p_\Gbb(\Zscr)=\Hscr\p p_0(\Gbb)$, for every $p\geq1$. This completes the proof of (i). We now come to (ii). If $Z\p1$, $Z\p2$ and $Z\p3$ are pairwise orthogonal, then $\Lscr\p2_\Gbb(\Zscr)=\Lscr_\Gbb\p2(Z\p1)\oplus\Lscr_\Gbb\p2(Z\p2)\oplus\Lscr_\Gbb\p2(Z\p3)$ (see \cite[Remark 4.36]{J79}). Then, since we have $Y-Y_0\in\Hscr\p2_0(\Gbb)$ and, by (i), $\Lscr\p2_\Gbb(\Zscr)=\Hscr\p2_0(\Gbb)$, we deduce (ii) from \cite[Theorem 4.6]{J79}. The proof of the corollary is complete.
\end{proof}

Regarding  $\Zscr=\{Z\p1,Z\p2,Z\p3\}\subseteq\Hscr\p1_{\mathrm{loc},0}$ as the \emph{multidimensional local martingale} $Z:=(Z\p1,Z\p2,Z\p3)\p \top$ and combining Corollary \ref{cor:st.ssp.Z} and \cite[Theorem 4.60]{J79}, yields that every $Y\in\Hscr\p p(\Gbb)$ can be represented as
\begin{equation}\label{eq:mul.in.rep}
Y=Y_0+K\cdot Z,\quad K\in\ol\Lrm\p p(Z,\Gbb),
\end{equation}
where $K\cdot Z$ is the \emph{multidimensional stochastic integral} and $\ol\Lrm\p p(Z,\Gbb)$ is defined as in \cite[Eq.\ (4.59)]{J79}.

\begin{remark}[Uniqueness]\label{rem:un}
Notice that the multidimensional stochastic integral is an isometric mapping between $\Hscr\p p_0(\Gbb)$ endowed with the equivalent norm $\|Y\|_{\Hscr\p p}\p\prime:=\Ebb[[Y,Y]_\infty\p{p/2}]\p{1/p}$ and $\ol\Lrm\p p(Z,\Gbb)$, $p\in[1,+\infty)$ (see \cite[(4.59) and the subsequent comment]{J79}).  Using this fact it is easy to show that the integrand $K$ in \eqref{eq:mul.in.rep} is unique in the $\ol\Lrm\p p(Z,\Gbb)$-norm. Similarly, one gets the uniqueness for the integrands $K\p i$, $i=1,2,3$, in Corollary \ref{cor:st.ssp.Z} (ii).
\end{remark}
\section{Predictable Representation with respect to Orthogonal Martingales}\label{subs:orth}
In this section we discuss some cases in which the martingale representation in $\Gbb$ yields with respect to orthogonal local martingales: In Subsection \ref{subs:in.en} we study the special case of the independent enlargement. In Subsection \ref{subs:enl.rt} we consider the case in which $\Fbb$ is progressively enlarged by a random time $\tau$, that is, $H=1_{\lsi \tau,+\infty\lsi}$ and we do not require the independence of the filtrations $\Fbb$ and $\Hbb$.

Of course, the local martingales $Z\p1$, $Z\p2$ and $Z\p3$ can be always orthogonalized in a standard way by the classical orthogonalization procedure for locally square integrable local martingales.

\subsection{The independent enlargement}\label{subs:in.en}
The filtrations $\Fbb$ and $\Hbb$ are now assumed independent, that is, the process $H$ is independent of $\{\Rscr,X\}$.
From \cite[Theorem 1]{WG82}, the filtration $\Fbb\vee\Hbb$ is right-continuous and so $\Gbb=\Fbb\vee\Hbb$ holds. Furthermore, by the independence, $\Fbb$-local martingales and $\Hbb$-local martingales remain $\Gbb$-local martingales. Hence, $\ol X\p\Fbb,\ol H\p\Hbb\in\Hscr\p1_\mathrm{loc}(\Gbb)$ and  $X\p{p,\Gbb}=X\p{p,\Fbb}$, $H\p{p,\Gbb}=H\p{p,\Hbb}$, $\ol X\p\Gbb=\ol X\p\Fbb$ and $\ol H\p\Gbb=\ol H\p\Hbb$ hold.

The proof of the following lemma is postponed to the appendix:

\begin{lemma}\label{lem:indep.orth} Let $\Fbb$ and $\Hbb$ be independent.
 Then the locally bounded local martingales $\ol X\p\Gbb$, $\ol H\p\Gbb$ and $[\ol X\p\Gbb,\ol H\p\Gbb]$ are pairwise orthogonal.
\end{lemma}

We now come to the main result of this section.
\begin{theorem}\label{thm:indep}
If $\Fbb$ and $\Hbb$ are independent, then every $Y\in\Hscr\p2(\Gbb)$ can be represented as follow:
\begin{equation}\label{eq:orth.ind}
Y=Y_0+K\p 1\cdot \ol X\p{\Gbb}+K\p 2\cdot \ol H\p{\Gbb}+K\p 3\cdot [\ol X\p{\Gbb},\ol H\p{\Gbb}]
\end{equation}
where $K\p 1\in\Lrm\p2(\ol X\p{\Gbb},\Gbb)$, $K\p 2\in\Lrm\p2(\ol H\p{\Gbb},\Gbb)$ and $K\p 3\in \Lrm\p2([\ol X\p{\Gbb},\ol H\p{\Gbb}],\Gbb)$. Additionally, \eqref{eq:orth.ind} is an orthogonal decomposition of $Y$ in $(\Hscr\p2(\Gbb),\|\cdot\|_2)$.
\end{theorem}
\begin{proof}
First, using \cite[Proposition I.4.49 (b)]{JS00}, we compute
\begin{equation}\label{eq:comp.sq.b.ind}
\begin{split}
[\ol X\p{\Gbb},\ol H\p{\Gbb}]&=[X, H]-[X\p{p,\Gbb},H]-\Delta H\p{p,\Gbb}\cdot\ol X\p\Gbb
\\&=[X, H]-[X\p{p,\Gbb},H\p{p,\Gbb}]+[X\p{p,\Gbb},H\p{p,\Gbb}]-[X\p{p,\Gbb},H]-\Delta H\p{p,\Gbb}\cdot\ol X\p\Gbb
\\&=[X, H]-[X\p{p,\Gbb},H\p{p,\Gbb}]-\Delta X\p{p,\Gbb}\cdot\ol H\p\Gbb-\Delta H\p{p,\Gbb}\cdot\ol X\p{\Gbb}
\end{split}
\end{equation}
Because of Lemma \ref{lem:indep.orth}, $[\ol X\p{\Gbb},\ol H\p{\Gbb}]\in\Mloc(\Gbb)$ holds. Hence, Lemma \ref{lem:orth.mart} (iii) yields the identity $[X,H]\p{p,\Gbb}=[X\p{p,\Gbb},H\p{p,\Gbb}]$. Thus, from \eqref{eq:mart.Z1.Z2.Z3} and the previous computation, we obtain
\begin{gather}
Z\p3=[\ol X\p{\Gbb},\ol H\p{\Gbb}]+\Delta X\p{p,\Gbb}\cdot\ol H\p\Gbb+\Delta H\p{p,\Gbb}\cdot\ol X\p{\Gbb}\label{eq:rep.Z3}\\
Z\p1=-[\ol X\p{\Gbb},\ol H\p{\Gbb}]-\Delta X\p{p,\Gbb}\cdot\ol H\p\Gbb+(1-\Delta H\p{p,\Gbb})\cdot\ol X\p{\Gbb}\label{eq:rep.Z1}\\
Z\p2=-[\ol X\p{\Gbb},\ol H\p{\Gbb}]+(1-\Delta X\p{p,\Gbb})\cdot\ol H\p\Gbb-\Delta H\p{p,\Gbb}\cdot\ol X\p{\Gbb}\label{eq:rep.Z2}.
\end{gather}
Let us denote $\Xscr:=\{\ol X\p\Gbb,\ol H\p\Gbb,[\ol X\p\Gbb,\ol H\p\Gbb]\}$.
We are now going to show $\Lscr\p2_\Gbb(\Zscr)\subseteq\Lscr_\Gbb\p2(\Xscr)$. We set $\Sscr:=\{\{(Z\p i)\p\sig,\ \sig\  \Gbb\textnormal{-stopping time such that the stopped process }\ (Z\p i)\p\sig\in\Hscr\p2(\Gbb)\}, i=1,2,3\}$. It is enough to verify that $\Sscr\subseteq\Lscr_\Gbb\p2(\Xscr)$. Indeed, by the properties of stable subspaces, $\Lscr\p2_\Gbb(\Zscr)$ is the smallest stable subspace of $\Hscr\p2_0(\Gbb)$ containing $\Sscr$.
So, we consider an arbitrary $\Gbb$-stopping time $\sig$ such that the stopped process $(Z\p 3)\p\sig$ belongs to $\Hscr\p2_0(\Gbb)$. Notice that such stopping times exist, $Z\p3$ being locally bounded. From \eqref{eq:rep.Z3}, since by Lemma \ref{lem:indep.orth} the locally bounded $\Gbb$-local martingales $\ol X\p\Gbb$, $\ol H\p\Gbb$ and $[\ol X\p{\Gbb},\ol H\p{\Gbb}]$ are pairwise orthogonal, we get
\[\begin{split}
&\aPP{(Z\p 3)\p\sig}{(Z\p 3)\p\sig}=\\&\qquad1_{\lsi0,\sig\rsi}\cdot\aPP{[\ol X\p{\Gbb},\ol H\p{\Gbb}]}{[\ol X\p{\Gbb},\ol H\p{\Gbb}]}+1_{\lsi0,\sig\rsi}(\Delta X\p{p,\Gbb})\p2\cdot\aPP{\ol H\p\Gbb}{\ol H\p\Gbb}+1_{\lsi0,\sig\rsi}(\Delta H\p{p,\Gbb})\p2\cdot\aPP{\ol X\p\Gbb}{\ol X\p\Gbb}.
\end{split}
\]
Since $(Z\p3)\p\sig\in\Hscr\p2_0(\Gbb)$, it follows that $\aPP{(Z\p 3)\p\sig}{(Z\p 3)\p\sig}\in\Ascr\p+(\Gbb)$. From this we obtain the following inclusions: $1_{\lsi0,\sig\rsi}\in\Lrm\p2([\ol X\p{\Gbb},\ol H\p{\Gbb}],\Gbb)$, $1_{\lsi0,\sig\rsi}\Delta X\p{p,\Gbb}\in \Lrm\p2(\ol H\p{\Gbb},\Gbb)$ and $1_{\lsi0,\sig\rsi}\Delta H\p{p,\Gbb}\in \Lrm\p2(\ol X\p{\Gbb},\Gbb)$. Hence,  we have $(Z\p3)\p\sig\in\Lscr\p2_\Gbb(\Xscr)$, for every $\Gbb$-stopping time $\sig$ such that $(Z\p 3)\p\sig$ belongs to $\Hscr\p2(\Gbb)$. Analogously, we see that the same holds also for $Z\p1$ and $Z\p2$. Hence, we get $\Sscr\subseteq\Lscr\p2_\Gbb(\Xscr)$ which yields $\Lscr\p2_\Gbb(\Zscr)\subseteq\Lscr\p2_\Gbb(\Xscr)$. From Corollary \ref{cor:st.ssp.Z} (i), we have $\Lscr\p2_\Gbb(\Zscr)=\Hscr\p2_0(\Gbb)$. Therefore, we obtain $\Lscr\p2_\Gbb(\Xscr)=\Hscr_0\p2(\Gbb)$. By the pairwise orthogonality  of the involved local martingales, it follows that $\Lscr\p2_\Gbb(\Xscr)=\Lscr\p2_\Gbb(\ol H\p{\Gbb})\oplus\Lscr\p2_\Gbb(X\p{\Gbb})\oplus\Lscr\p2_\Gbb([\ol X\p{\Gbb},\ol H\p{\Gbb}])$ (see \cite[Remark 4.36]{J79}). This, recalling \cite[Theorem 4.6]{J79}, shows \eqref{eq:orth.ind}. The proof of the theorem is now complete.
\end{proof}
\paragraph*{The multiplicity of $\Gbb$ in the independent enlargement.}\label{rem:mul.indep}
The multiplicity (see Davis and Varaiya \cite{dv:mi}) or the spanning number (see Duffie \cite{du:se}) of a filtration is the minimal number of locally square integrable orthogonal local martingales which is necessary to represent all square integrable martingales (adapted to the filtration).

Theorem \ref{thm:indep} shows that, in general, the multiplicity of the filtration $\Gbb$ is less than or equal to \emph{three}. For an example in which the multiplicity of $\Gbb$ is equal to three we refer to \cite[Theorem 2.3]{CT17a} without Brownian part and applied to the following case: $X=1_{\lsi\et,+\infty\lsi}$ and $H=1_{\lsi\tau,+\infty\lsi}$, where $\et$ and $\tau$ are two \emph{independent} random variables taking values in $\{1,2,3\}$ and $\{2,4\}$ with strictly positive probability, respectively.

However, the multiplicity of $\Gbb$ can become less than or equal to two: For example, according to \cite[Lemma 7]{WG82}, under the independence assumption, if $\Delta X\p{p,\Fbb}\Delta H\p {p,\Hbb}=0$, then $\Delta\ol X\p\Gbb\Delta\ol H\p\Gbb=0$. Hence, $[\ol X\p\Gbb,\ol H\p\Gbb]=0$. This is the case, e.g., if $X$ is quasi-left continuous with respect to $\Fbb$ (or if $H$ is quasi-left continuous with respect to $\Hbb$).

Additionally, there are situations in which the multiplicity of $\Gbb$ is equal to one. Indeed, because of \cite[Theorem 3]{WG82}, this is the case if and only if the $\Gbb$-local martingale $\ol X\p\Fbb+\ol H\p\Hbb=\ol X\p\Gbb+\ol H\p\Gbb$ has the PRP with respect to $\Gbb$ or, equivalently, if and only if $X\p{p,\Fbb}$ and $H\p{p,\Hbb}$ are mutually singular with respect to $\Pscr(\Gbb)$, i.e., if there exists a $\Gbb$-predictable set $A$ such that $X\p{p,\Gbb}=1_A\cdot X\p{p,\Gbb}$ and  $H\p{p,\Gbb}=1_{A\p c}\cdot H\p{p,\Gbb}$.

It is easy to construct an example in which the multiplicity of $\Gbb$ is equal to one. Indeed, let $C$ be the monotonic and continuous extension of the Cantor function $c$ on $[0,1]$ to the whole positive half line, that is, $C(n+s):=n+c(s)$, $s\in[0,1]$, $n\geq0$. Then, $C$ is increasing, continuous, singular and satisfies $C_0=0$, $C_\infty=+\infty$. In particular, we can consider $C$ as a \emph{continuous time change}. Let now $H$ and $N$ be two independent homogeneous Poisson processes. We define $X_t:=N_{C_t}$, $t\geq0$. For simplicity, let the initial $\sig$-field $\Rscr$ be trivial, thus $\Fbb=\Xbb$. By the properties of continuous time changes (see, e.g., \cite[Theorem 10.17 (c)]{J79}) we have $X\p{p,\Fbb}_t=C_t$. This means that $X$ is a \emph{non-homogeneous} Poisson process. Furthermore, since $H\p{p,\Hbb}_t=t$, the $\Gbb$-dual predictable projections $X\p{p,\Gbb}=X\p{p,\Fbb}$ and $H\p{p,\Gbb}=H\p{p,\Hbb}$ are mutually singular on $\Pscr(\Gbb)$ and by \cite[Theorem 3]{WG82} the multiplicity of $\Gbb$ is equal to one.

\subsection{The Progressive Enlargement by a Random Time}\label{subs:enl.rt}
In this subsection we do not assume that the random time $\tau$ and the filtration $\Fbb$ are independent. We rather
state the following assumptions:
\begin{assumption}\label{ass:susbs}\indent
(1) We have $H=1_{\lsi\tau,+\infty\lsi}$, where $\tau:\Om\longrightarrow(0,+\infty]$ is a \emph{random time}.

(2) For every finite-valued $\Fbb$-predictable stopping time $\sig$ we have $\Delta X_\sig\Delta H_\sig=0$ a.s.
\end{assumption}

We stress that Assumption \ref{ass:susbs} (2) is satisfied if, for example, $X$ is quasi-left continuous with respect to $\Fbb$ or if $\tau$ avoids $\Fbb$-predictable stopping times, that is $\Pbb[\tau=\sig<+\infty]=0$, for every $\Fbb$-predictable time $\sig$. We notice that $\tau$ avoids $\Fbb$-predictable stopping times if and only if $H\p{p,\Gbb}$ is continuous, meaning that $\tau$ is $\Gbb$-totally inaccessible (see \cite[p.65]{Jeu80}).

Moreover, we observe that, if  Assumption \ref{ass:susbs} (1) holds, the quasi-left continuity of $X$ with respect to $\Fbb$ is preserved in $\Gbb$ only \emph{up to time} $\tau$. Indeed, by \cite[Corollary 2.17]{AJ17}, for every $\Gbb$-predictable stopping time $\et$ there exists an $\Fbb$-predictable stopping time $\sig$ such that $\et\wedge\tau=\sig\wedge\tau$. Hence, let $\et$ be a $\Gbb$-predictable stopping time. By \cite[Corollary 2.17]{AJ17} and the quasi-left continuity of $X$ with respect to
$\Fbb$, we have
\[
\Delta X\p\tau_\et=\Delta X_{\eta\wedge\tau}=\Delta X_{\sig\wedge\tau}=(\Delta X_\sig)\p\tau=0,
\]
where $\sig$ is an $\Fbb$-predictable stopping time. This shows that $X\p\tau$ is quasi-left continuous with respect to $\Gbb$. In  Counterexample \ref{cex:no.qlc} below we show that the $\Gbb$-quasi-left continuity of $X$ is, in general, not preserved after $\tau$.

We stress that, by Lemma \ref{lem:orth.mart} (v), $Z\p i$ and $Z\p j$ are orthogonal if and only if $[Z\p i,Z\p j]=0$, $i,j=1,2,3$, $i\neq j$.
\begin{theorem}\label{thm:orth}
Let Assumption \ref{ass:susbs} (1) hold.

\textnormal{(i)}  If  Assumption \ref{ass:susbs} (2) holds, then $[Z\p i,Z\p 3]=0$, $i=1,2$.

\textnormal{(ii)} If  $H\p{p,\Gbb}$  or $X\p{p,\Gbb}$ are continuous, then $[Z\p i,Z\p j]=0$, $i,j=1,2,3$, $i\neq j$.

\textnormal{(iii)} If $X$ is quasi-left continuous with respect to $\Fbb$ (i.e.\ $X\p{p,\Fbb}$ is continuous), then $[(Z\p1)\p\tau,Z\p j]=[Z\p2,Z\p3]=0$, $j=2,3$.
\end{theorem}
\begin{proof} First we observe that, by Assumption \ref{ass:susbs} (2), the process $U:=[X,H]$ is quasi-left continuous with respect to $\Gbb$.  Indeed, for any finite-valued $\Gbb$-predictable stopping time $\et$, since $U=U\p\tau$, by \cite[Corollary 2.12]{AJ17}, we have $\Delta U_\et=(\Delta X_\et\Delta H_\et)\p\tau=0$ a.s. Hence, $U\p{p,\Gbb}$ is continuous.

We now verify (i). We define the processes $Y=X-U$. Then $\Delta Y\Delta U=0$ and $\Delta Y\p{p,\Gbb}\Delta U\p{p,\Gbb}=0$. Applying Lemma \ref{lem:orth.mart} (iv) and (v) to $Y$ and $Z=U$ we deduce $[Z\p 1,Z\p 3]=0$. Hence, $Z\p1$ and $Z\p 3$ are orthogonal. The proof for $Z\p2$ is completely analogous. We now come to (ii). If $H\p{p,\Gbb}$ is continuous, then Assumption \ref{ass:susbs} (2) holds. Thus, by (i) it is sufficient to verify  $[Z\p1,Z\p 2]=0$. By the continuity of  $U\p{p,\Gbb}$, the process $(H-U)\p{p,\Gbb}$ is continuous as well. Hence, Lemma \ref{lem:orth.mart} (v) with $Y=Z\p1$ and $Z=Z\p2$ yields the result, since $X-U$ and $H-U$ have no common jumps. If $X\p{p,\Gbb}$ is continuous the proof is analogous. We now come to (iii). By the quasi-left continuity of $X$ with respect to $\Fbb$, Assumption \ref{ass:susbs} (2) is satisfied and $X\p\tau$ is quasi-left continuous with respect to $\Gbb$, i.e., $(X\p\tau)\p{p,\Gbb}$ is continuous. By $[(Z\p 1)\p \tau,Z\p3]=[Z\p 1,Z\p3]\p \tau$ and (i), it is enough to verify $[(Z\p1)\p\tau, Z\p2]=0$. Because of $(X\p{p,\Gbb})\p\tau=(X\p\tau)\p{p,\Gbb}$, we have $((X-U)\p\tau)\p{p,\Gbb}=(X\p{p,\Gbb})\p\tau-U\p{p,\Gbb}$, which is a continuous process. Since $Y=X\p\tau-U$ and $Z=H-U$ have no common jumps and $\Delta Y\p p\Delta Z\p p=0$, by Lemma \ref{lem:orth.mart} (v), we deduce the claim. The proof of the theorem is complete.
\end{proof}
We remark that, as shown in Counterexample \ref{cex:jum.com} below, Assumption \ref{ass:susbs} (2) alone is not sufficient to ensure that $Z\p1$ and $Z\p2$ are orthogonal. Indeed, it could happen, in general, that $(X-[X,H])\p{p,\Gbb}$ and $(H-[X,H])\p{p,\Gbb}$ have common jumps, although $X-[X,H]$ and $H-[X,H]$ do not have common jumps.

\begin{corollary}\label{cor:orth.av}
Let $\tau$ avoid $\Fbb$-stopping times. Then $Z\p 3=0$, $Z\p1=\ol X\p{p,\Gbb}$, $Z\p2=\ol H\p{p,\Gbb}$ and $[Z\p1,Z\p2]=0$.
\end{corollary}
\begin{proof}
By the avoidance assumption,  $Z\p 3=0$ holds and $H\p{p,\Gbb}$ is continuous. Theorem \ref{thm:orth} (ii) yields $[Z\p1,Z\p2]=0$. The proof is complete.
\end{proof}
Let $Y\in\Hscr\p2(\Gbb)$ and let $X$ be quasi-left continuous with respect to $\Fbb$.
From Corollary \ref{cor:st.ssp.Z}, Theorem \ref{thm:orth} (iii) and the properties of the multidimensional stochastic integral, we get
\begin{equation}\label{eq:rep.stopped}
Y\p\tau=Y_0+K\p1\cdot (Z\p1)\p\tau+K\p2\cdot Z\p2+K\p3\cdot Z\p3, \quad K\p i\in\Lrm\p2(Z\p i,\Gbb),\quad i=1,2,3,
\end{equation}
and this is an orthogonal representation in $(\Hscr\p2(\Gbb),\|\cdot\|_2)$.

\begin{remark}[The multiplicity of $\Gbb$]
Because of Corollary \ref{cor:st.ssp.Z} and Theorem \ref{thm:orth} (ii) we see that, if $H\p{p,\Gbb}$ is continuous, then the local martingales $Z\p1$, $Z\p2$ and $Z\p3$ are orthogonal and the multiplicity of the filtration $\Gbb$ is less than or equal to three. From Corollary \ref{cor:orth.av}, we see that, if $\tau$ avoids $\Fbb$-stopping times then the multiplicity of $\Gbb$ is less than or equal to two. Furthermore, it is possible to show that $X\p{p,\Gbb}$ and $H\p{p,\Gbb}$ are mutually singular on $\Pscr(\Gbb)$ if and only if the multiplicity of $\Gbb$ is equal to one, even without assuming the independence of $\tau$ and $\Fbb$ (see \cite[Theorem 5.4]{DTE20}). Finally, from \eqref{eq:rep.stopped}, if $X$ is quasi-left continuous with respect to $\Fbb$, we see that the space $\Hscr\p2_\tau(\Gbb)$ of square integrable martingales \emph{stopped at $\tau$} has always an integral representation with respect to three orthogonal martingales. So, following \cite[Remark 3.3]{AJR18}, we can say that the multiplicity of the class $\Hscr\p2_\tau(\Gbb)$ in $\Gbb$ is, in general, less than or equal to three. However, the multiplicity of the class $\Hscr\p2_\tau(\Gbb)$ can reduce to two or even to one.
\end{remark}
\paragraph*{The quasi-left continuity of $X$.} We conclude this subsection with a short discussion about the quasi-left-continuity of $X$ with respect to $\Fbb$.

The quasi-left continuity of an adapted \cadlag process is a property of the process but also of the filtration: In the enlargement of a filtration there will be more predictable times and hence we cannot expect that the quasi-left continuity is preserved. The following simple example illustrates this fact.
\begin{cexample}\label{cex:no.qlc}  Let $X$ be a homogeneous Poisson process with respect to $\Xbb$ and let $(\tau_n)_{n\geq 1}$ be the sequence of the jump-times of $\Xbb$. The process $X$ is not quasi-left continuous in the filtration $\Fbb$ obtained enlarging $\Xbb$ initially by the $\sig$-field $\Rscr=\sig(\tau_1)$. Indeed, $\tau_1$ is a jump-time of $X$ and, in $\Fbb$, it is $\Fscr_0$-measurable. Hence, $(\sig_n)_{n\geq1}$, $\sig_n:=(1-\frac1{2n})\tau_1$ is a sequence of $\Fbb$-stopping times announcing $\tau_1$. Therefore, $\tau_1$ is an $\Fbb$-predictable jump-time of $X$. Moreover, $X$ is not quasi-left continuous in the filtration $\Gbb$ obtained enlarging $\Xbb$ progressively by the random time $\tau=\frac{1}{2}(\tau_1+\tau_2)$. Indeed, the jump-time $\tau_2$ of $X$ is announced in $\Gbb$ by $(\vartheta_n)_{n\geq1}$, $\vartheta_n:=\frac1n\tau+(1-\frac1n)\tau_2$, and $\vartheta_n>\tau$ is a $\Gbb$-stopping time for every $n\geq1$ by \cite[Theorem III.16]{D72}. Hence, $\tau_2$ is a $\Gbb$-predictable jump-time of $X$.
\end{cexample}

We have assumed in Theorem \ref{thm:orth} (iii) that the process $X$ is quasi-left continuous with respect to the filtration $\Fbb$, obtained enlarging $\Xbb$ initially by the $\sig$-field $\Rscr$.
However, in general, it is more interesting to start assuming the quasi-left continuity of $X$ with  respect to $\Xbb$ and then to initially enlarge $\Xbb$ by $\Rscr$ in such a way to preserve the quasi-left continuity of $X$ in $\Fbb$. As an example in which this is possible, we consider the case of a quasi-left continuous point process $X$ with respect to the filtration $\Xbb$ and then we enlarge $\Xbb$ initially by a nonnegative random variable $L$ (i.e., $\Rscr=\sig(L)$) which satisfies \emph{Jacod's absolute continuity hypothesis} (see, e.g., \cite[Section 4.4]{AJ17}). Indeed, in this special but important case, from \cite[Theorem 4.25]{AJ17} one can show  that $X\p{p,\Fbb}$ is continuous or, equivalently, that $X$ is quasi-left continuous with respect to $\Fbb$. In this case, $X\p\tau$ is quasi-left continuous with respect to $\Gbb$ but this need not be true for $X$. If $\Fbb$ is immersed in $\Gbb$, then $X\p{p,\Fbb}=X\p{p,\Gbb}$ and $X$ is quasi-left continuous with respect to $\Gbb$.
\subsection{Two concrete examples}\label{subs:ex}
Let $X$ be a homogeneous Poisson process with respect to $\Xbb=(\Xscr_t)_{t\geq0}$ and let $(\tau_n)_{n\geq 1}$ be the sequence of the jump-times of $X$. We denote $\lambda$ the intensity of $X$. For a random time $\tau$, we set   $H=1_{\lsi\tau,+\infty\lsi}$. We give two examples of the progressive enlargement of the filtration $\Xbb$ by $\Hbb$, that is, $\Rscr$ is assumed trivial and $\Xbb=\Fbb$.

\paragraph*{Progressive enlargement by $X_T+1$.} We consider an arbitrary but fixed time $T>0$ and set $\tau = X_T+1$. Then, $\tau$ takes values in $(0,+\infty)$ and $H$ is a point process. In particular, $H_0=0$ holds. As we are going to show in Remark \ref{rem:not.hon} below, $\tau$ is not a honest time. Hence, this example cannot be recovered by Barlow \cite{ba:sf}. We furthermore stress that the filtration $\Xbb$ is not immersed in $\Gbb$. Indeed, the condition $\Pbb[\tau>t|\Xscr_t]=\Pbb[\tau>t|\Xscr_\infty]$ (which according to \cite[Lemma 3.8]{AJ17} is equivalent to the immersion property) is evidently not satisfied. Furthermore, the random time $\tau$ does not avoid all $\Xbb$-(predictable) stopping times since $\Pbb[\tau=n]>0$ holds, for every $n\geq1$. Hence, this example cannot be recovered by \cite{DT19}. In other words, it seems that the present example is new and can be only studied thanks to the results obtained in this paper.

We start noticing that $\tau$ avoids the jump-times of $X$ since $\tau_n$ is Gamma distributed and hence
\[
\Pbb[\tau=\tau_n]=\sum_{k=1}\p\infty\Pbb[\tau_n=k|\tau=k]\Pbb[\tau=k]=0,\quad n\geq1.
\]
Therefore, the set $\lsi\tau\rsi\cap\lsi\tau_n\rsi$ is evanescent, for every $n\geq1$ and, since   $[X,H]=\Delta X_\tau H$, we get $[X,H]=0$. Hence, we deduce that $Z\p3=0$, $Z\p 1=\ol X\p\Gbb$ and $Z\p2=\ol H\p\Gbb$ hold. We notice that $Z\p 1$ and $Z\p 2$ are orthogonal. To see this, we are going to explicitly compute the $\Gbb$-dual predictable projections $X\p{p,\Gbb}$ and $H\p{p,\Gbb}$ and apply Theorem \ref{thm:orth}.

We first compute $H\p{p,\Gbb}$. To this end, we denote by $\xi$ the law of $\tau$ and by $(u,x)\mapsto h(u,x)$ the function defined by $h(u,k):=\rme^{-\lambda u} \frac{(\lambda u)^k}{k!}$ if $k=1,\ldots$ and $u\geq0$, and $h(u,k):=0$ else. We observe that  $\Pbb[\tau=k]= h(T,k-1)$, $k\geq1$ and that
 \[
\xi(\rmd u)=\sum_{k=1}\p\infty h(T,k-1)\delta_k(\rmd u).
\]
 We denote by $P_t(\cdot,A)$ a regular version of $\Pbb[\tau\in A|\Xscr_t]$, $A\in\Bscr(\Rbb)$. It is then easy to verify (see \cite[Example 4.15]{AJ17}) that $P_t(\rmd u)$ is absolutely continuous with respect to $\xi$
 and that a version of the density is
\[
\alpha_t(u)=\sum_{k= 1}\p\infty \frac{z^k_t}{h(T,k-1)} 1_{\{u=k\}},
\]
where we denote by $z\p k$ the martingale satisfying
\[z^k_t=\Pbb [\tau=k\vert \cX_t]=\Pbb[X_T=k-1 \vert \cX_t]\quad \text{a.s.}\quad t\geq0.\]
In other words, $\tau$ satisfies \emph{Jacod's absolutely continuity condition} and, according to \cite[\S 5.3]{AJ17}, we say that $\tau$ is a $\Jscr$-time. Because of the independence of the increments of $X$ with respect to $\Xbb$, we get
\[z^k_t=\Pbb[X_T-X_t=k-1-X_t \vert \cX_t] =\begin{cases}h(T-t,k-1-X_t),\quad &t<T\\1_{\{X_T=k-1\}},\quad &t\geq T.\end{cases}\]
From \cite[Corollary 5.27 (b)]{AJ17}, the  $\Xbb$-dual predictable projection of $H$ is  $H\p{p,\Xbb}_t=\int_0^ t\alpha^u_{u-}(du)\xi(\rmd u)$ and
\[H\p{p,\Gbb}_t  =\int_0^{t\wedge\tau}\frac{1}{Z_{u-}}\alpha_{u-}(u)\xi(\rmd u)=\sum_{u\leq t\wedge\tau}\frac{1}{Z_{u-}}\alpha_{u-}(u)\xi(\{u\})\,,\]
where, $Z_t:=\sum_{k=1}^\infty 1_{\{t<k\}}z^k_t$ denotes the Az\'ema supermartingale. In particular, we notice that $H\p{p,\Gbb}$ is a purely discontinuous increasing process.

We now come to the computation of $X\p{p,\Gbb}$. To this aim, we first notice that, by \cite[Corollary 5.27 (b)]{AJ17},  the $\Xbb$-dual optional projection of $H$, say $H\p{o,\Xbb}$, is given by
\[
H\p{o,\Xbb}_t=\int_0^ t\alpha_{u}(u)\xi(\rmd u)=\sum_{u\leq t}\al_u(u)\xi (\{u\}).
\]
Hence, $H\p{o,\Xbb}$ is a purely discontinuous increasing process. By \cite[Theorem 1.43 (b)]{AJ17}, we deduce that $\tau$ is a \emph{thin} random time (see \cite[Definition 1.40]{AJ17}). Therefore, according to \cite[Proposition 5.33]{AJ17}, the $\Gbb$-special semimartingale decomposition of $\ol X\p\Xbb_t=X_t-\lm t$ is given by
\[
\ol X\p\Xbb_t= \widehat M_t + \int_0^{t\wedge \tau}\frac{1}{Z_{s-}}\rmd\cro{\ol X\p\Xbb,m}_s +\sum_{k=1}^\infty 1_{\{\tau=k\}} \int_0^t
1_{\{s>k\}}\frac{1}{z^k_{s-}}\rmd \cro{\ol X\p\Xbb,z^k}_s,
\]
where $\widehat M$ is a $\Gbb$-local martingale 
and $m$ is the $\Xbb$-martingale defined by $m_t:=\sum_{k=1}^\infty z^k_{t \wedge \tau_k}$. We also stress that the predictable brackets $\cro{\ol X\p\Xbb,m}$ and $ \cro{\ol X\p\Xbb,z^n}$ are computed with respect to the filtration $\Xbb$. We note that $\cro{\ol X\p\Xbb,m}=  \sum_{k=1}^\infty\cro{\ol X\p\Xbb, z^k}$. Hence, to obtain a closed formula of the $\Gbb$-dual predictable projection  $X\p{p,\Gbb}$, it remains to compute $\cro{\ol X\p\Xbb, z^k}$. To this end, we observe that
\[
\begin{split}\rmd z^k_t&=1_{\{t<T\}} \big[ \partial _t h(T-t,k-1-X_t) \rmd t +h(T-t, k-1-X_t)-h(T-t, k-1-X_{t-})\Big]\\&=1_{\{t<T\}}\Big[ \partial _t h(T-t, k-1-X_t) \rmd t +\big(h(T-t, k-2-X_{t-})-h(T-t, k-1-X_{t-})\big)\rmd X_t\Big]\\
 &=1_{t<T} \big(h(T-t,k-2-X_{t-})-h(T-t, k-1-X_{t-}) \big) \rmd \ol X\p \Xbb_t\,,
\end{split}
\]
where we used It\^o's calculus  to get the first equality and computed explicitly $\partial _u h(u,k) $ to obtain the third  equality. Finally, we obtain
 \[\rmd\cro{\ol X\p\Xbb,z^k}_t=  1_{t<T}\big(h(T-t, k-2-X_{t-})-h(T-t, k-1-X_{t-}) \big)\lambda \rmd t.\]\\
It follows that
\[ X^{p,\Gbb}_t= \lm t + \int_0^{t\wedge \tau}\frac{1}{Z_{s-}}\rmd \cro{\ol X\p\Xbb,m}_s +\sum_{k=1}^\infty 1_{\{\tau=k\}} \int_0^t
1_{\{s>k\}}\frac{1}{z^k_{s-}}\rmd \cro{\ol X\p\Xbb,z^k}_s\]
 and a closed-form formula can be obtained (we do not give details). We only stress that $ X^{p,\Gbb}$ is absolutely continuous with respect to the Lebesgue measure. Hence, since $[X,H]=0$, we can apply Theorem \ref{thm:orth} to deduce that $Z\p1=\ol X\p\Gbb$ and $Z\p2=\ol H\p{\Gbb}$ are orthogonal.

We now discuss the multiplicity of $\Gbb$ in this example. The local martingale $Z\p 1, Z\p2\in\Hscr\p2_\mathrm{loc}(\Gbb)$ are orthogonal and possess the PRP with respect to $\Gbb$, because of Theorem \ref{thm:wrp.srp} (ii). Furthermore, $X\p {p,\Gbb}$ is absolutely continuous, while $H\p{p,\Gbb}$ is a purely discontinuous increasing process. By Lemma \ref{lem:orth.mart} (vi), we see that $\aPP{Z\p 1}{Z\p 1}$ is again absolutely continuous and  $\aPP{Z\p 2}{Z\p 2}$ is purely discontinuous. So, these processes are mutually singular with respect to $\Pscr(\Gbb)$. Therefore, by \cite[Theorem 5.4]{DTE20}, the $\Gbb$-local martingale $Z\p 1+Z\p2$ has the PRP with respect to $\Gbb$ and the multiplicity of $\Gbb$ is equal to one.

Finally, we stress that taking $\tau=X_T$ will lead to similar but a bit more involved computations. Indeed, in this latter case $\Pbb[\tau=0]>0$ and $H$ is not a point process (in the sense of \cite[Section I.3(b).3]{JS00}), because $H_0\neq0$. However, we can consider $H\p\prime:=H-H_0$ which is a point process and set $\Rscr=\{\tau=0\}$. In this way, we get $\Hbb=\Hbb\p\prime\vee\Rscr$, where $\Hbb\p\prime$ denotes the filtration generated by $H\p\prime$, and $\Xbb\vee\Hbb\p\prime\vee\Rscr=\Xbb\vee\Hbb$. The dual predictable projection of $H\p\prime$ must be now computed with respect to $\Hbb$. In other words, one has to look at this case as an initial enlargement of $\Hbb\p\prime$ by $\Rscr$ and a progressive enlargement of $\Xbb$ by $\Hbb$.

\begin{remark}\label{rem:not.hon}
We now prove that $\tau$ is not a honest random time. Let  $\wt Z$ be the c\`adl\`ag $\Xbb$-supermartingale such that $\wt Z_\sig=\Pbb[\tau\geq\sig|\Xscr_\sig]$, for every finite-valued  $\Xbb$-stopping time $\sig$. According to \cite[Theorem 8.8 (b)]{AJ17}, $\tau$ is honest if and only if $\wt Z_\tau=1$. We are going to prove that $\wt Z_\tau<1$ with strictly positive probability.

From \cite[Proposition 2.5]{ACJ}, we get $ \wt Z_\tau=\sum_{n=1}^\infty 1_{\{\tau \leq n\}}z^n_\tau$.
It follows that, for $\tau<T$ and each $k\geq0$, 
\[\begin{split}
\wt Z_\tau 1_{\{\tau=k\}}&=1_{\{\tau=k\}}\sum_{n \geq k}z^n_k=1_{\{\tau=k\}}\sum_{n \geq k}h(T-k, n-1-X_k)=1_{\{\tau=k\}}\rme^{-\lambda (T-k)}\sum_{n \geq k} \frac{(\lambda (T-k))^{n-1-X_k}}{(n-1-X_k)!}
\end{split}\] and on $\{\tau=k\leq n\} \cap \{k<T\}$,  one has $X_k\leq  X_T=k-1\leq n-1$  so that $n-1-X_k\geq 0$. Setting $m= n-k$ yields
\[\sum_{m \geq 0} \frac{(\lambda (T-k))^{m+k-1-X_k}}{(m+k-1-X_k)!}\leq \sum_{n \geq 0} \frac{(\lambda (T-k))^{n}}{n !}= \rme^{\lambda (T-k)}\]
with a strict inequality if, for $m=0$, the left-hand side is not equal to 1 a.s., that is, if there exists $k$ such that $\Pbb [k-1-X_k=0] \neq 1$. For $k=2<T$ and on the set $\{\tau_1>2\}$ (which has a strictly positive probability) it is obvious that $k-1-X_k=1$ and we are done.
 Hence $\Pbb[\wt Z_\tau = 1]<1$       and $\tau$ is not   honest.
\end{remark}
\paragraph*{Progressive enlargement by the minimum.} Here, we consider  $\tau = \tau_1 \wedge a\tau_2$ with $0<a<1$. In this case $[X,H]_t= 1_{\{\tau_1 \leq t \}}1_{\{\tau_1<aT_2\}}\neq 0$ and we expect the multiplicity of $\Gbb$ to be, in general, less than or equal to three. However, we are not able to  give explicit form $\Gbb$-dual predictable projections of the processes $X$, $H$ and $ [X,H]$, since $\tau$ is nor thin neither satisfies Jacod's absolutely continuity conditions. Hence, we did not succeed in giving an explicit form of the $\Gbb$-local martingales  $Z\p i$, $i =1,2,3$. Notice that, in this example, $\tau$ is not honest, is not a $\Jscr$-time, is not thin, it does not avoid $\Xbb$-stopping times and $\Xbb$ is not immersed in $\Gbb$. Nevertheless, thanks to Theorem \ref{thm:wrp.srp}, we know that $(X,H)$ has the WRP and that $Z\p 1$, $Z\p 2$ and $Z\p 3$ have the PRP with respect to $\Gbb$.

\appendix
\section{Some technical proofs and results}\label{app:te.pr.re}

\begin{proof}[Proof of Proposition \ref{prop:poi.pr.pb}]
Since $X$ and $H$ are $\Gbb$-point processes, using the definition of the quadratic variation, we get $X-[X,H]=\sum_{0\leq s\leq\cdot}\Delta X_s (1-\Delta H_s)$, from which it immediately follows that $X-[X,H]$ is a point process with respect to $\Gbb$. Analogously, $H-[X,H]$ and $[X,H]$ are $\Gbb$-point processes. It is clear that these processes have pairwise no common jumps. Indeed, for example, $[X-[X,H],[X,H]]=[X,[X,H]]-[[X,H],[X,H]]=[X,H]-[X,H]=0$.
\end{proof}

\begin{proof}[Proof of Lemma \ref{lem:fil.eq}]By definition, $\Gbb$ is the smallest right-continuous filtration satisfying $\Rscr\subseteq\Gscr_0$ and containing $\Xbb$ and $\Hbb$. Therefore, since $\wt X$ is a $\Gbb$-semimartingale, $\mu\p{\wt X}$ is $\Gbb$-optional and hence $\Gbb\p\ast\subseteq\Gbb$, by the definition of $\Gbb\p\ast$. Moreover, $\Gbb\subseteq\Gbb\p\ast$ holds. Indeed, from \eqref{eq:ju.mea.spp}, we have the identities
$X=(1_{\{x_1=1,x_2=0\}}+1_{\{x_1=1,x_2=1\}})\ast\mu\p{\wt X}$ and $H=(1_{\{x_1=0,x_2=1\}}+1_{\{x_1=1,x_2=1\}})\ast\mu\p{\wt X}$,
showing that $X$ and $H$ are $\Gbb\p\ast$-optional processes, $\mu\p{\wt X}$ being $\Gbb\p\ast$-optional. Thus, $\Gbb=\Gbb\p\ast$and the proof is complete.
\end{proof}

\begin{proof}[Proof of Lemma \ref{lem:indep.orth}]
We only verify that $\ol X\p\Gbb\ol H\p\Gbb\in\Mloc(\Gbb)$, i.e., that $[\ol X\p\Gbb,\ol H\p\Gbb]\in\Mloc(\Gbb)$. For the pairwise orthogonality we refer to the proof of \cite[Proposition 3.4]{CT19}. So, let $(\tau_n)_{n\geq0}$ be a sequence of $\Fbb$-stopping times localizing $\ol X\p\Gbb=\ol X\p\Fbb$ to $\Hscr\p2(\Fbb)$ and $(\sigma_n)_{n\geq0}$ a sequence of $\Hbb$-stopping times localizing $\ol H\p\Gbb=\ol H\p{\Hbb}$ to $\Hscr\p2(\Hbb)$. By the independence of $\Fbb$ and $\Hbb$, the process $(\ol X\p\Gbb)\p{\tau_n}(\ol H\p\Gbb)\p{\sig_n}$ belongs to $\Hscr\p2(\Gbb)$. Therefore, setting $\rho_n:=\sig_n\wedge\tau_n$, we find a sequence of $\Gbb$-stopping times localizing $\ol X\p\Gbb\ol H\p\Gbb$ to $\Hscr\p2(\Gbb)$. Concerning the locally boundedness, it is enough to observe that $\ol X\p\Gbb$ and $\ol H\p\Gbb$ are locally bounded, since, being   compensated point processes, they have bounded jumps. Therefore, $[\ol X\p\Gbb,\ol H\p\Gbb]$ has bounded jumps and it is  locally bounded as well. The proof is complete.
\end{proof}

\noindent\textbf{A lemma on compensated point processes.}
Let $\Abb$ be an \emph{arbitrary} right-continuous filtration. We recall that two processes $A$ and $B$ in $\Ascr\p+_\mathrm{loc}(\Abb)$ are called \emph{associated} if they have the same compensator, that is $A\p p=B\p p$. Clearly, $A$ and $B$ are associated if and only if $A-B\in\Mloc$ (see \cite[Theorem V.38]{D72}).
\begin{lemma}\label{lem:orth.mart}
Let $Y$ and $Z$ be point processes with respect to $\Abb$ and let $Y\p p$ and $Z\p p$ denote the  $\Abb$-dual predictable projection of $Y$ and $Z$, respectively. Let us denote $\ol Y:=Y-Y\p p$ and $\ol Z:=Z-Z\p p$.

\textnormal{(i)} The processes $[Y\p p,Z]$, $[Y,Z\p p]$ and $[Y\p p,Z\p p]$ are locally bounded and belong to $\Ascr\p+_\mathrm{loc}(\Abb)$.

\textnormal{(ii)} The processes $[Y\p p,Z]$ and $[Y,Z\p p]$ are associated and their compensator is $[Y\p p,Z\p p]$.

\textnormal{(iii)} The process $[\ol Y,\ol Z]$ belongs to $\Mloc(\Abb)$ if and only if $[Y,Z]\p p=[Y\p p,Z\p p]$.

\textnormal{(iv)} Let $\Delta Y\Delta Z=0$. Then $[\ol Y,\ol Z]\in\Mloc(\Abb)$ if and only if $[\ol Y,\ol Z]=0$.

\textnormal{(v)} Let $\Delta Y\Delta Z=0$. Then $[\ol Y,\ol Z]\in\Mloc(\Abb)$ if and only if $\Delta Y\p p\Delta Z\p p=0$.

(vi) The identity $\aPP{\ol Y}{\ol Y}=(1-\Delta Y\p p)\cdot Y\p p$ holds.
\end{lemma}
\begin{proof}
We first verify (i). We only show that $[Y, Z\p p]$ is a locally bounded  increasing process, the proof for $[Y\p p, Z]$ and $[Y\p p, Z\p p]$ being completely analogous. We have $\Delta [Y, Z\p p]=\Delta Y\Delta Z\p p\geq0$, because $Y$ and $Z\p p$ are both increasing. Since $[Y,Z\p p]=\sum_{s\leq\cdot}\Delta Y_s\Delta Z_s\p p$, we obtain that $[Y,Z\p p]$ is an increasing process. Furthermore, since  $Y$ and $Z\p p$ have bounded jumps, $[Y, Z\p p]$ has bounded jumps too. Hence, it is a locally bounded process. The proof of (i) is complete.
 We now come to (ii). By \cite[Proposition I.4.49 (a)]{JS00} we have $[Y\p p,Z]=\Delta Y\p p\cdot Z$ and $[Y,Z\p p]=\Delta Z\p p\cdot Y$. Then, since $\Delta Y\p p$ is a predictable process, we have $[Y\p p,Z]\p p=(\Delta Y\p p\cdot Z)\p p=\Delta Y\p p\cdot Z\p p=[Y\p p,Z\p p]$, where in the last equality we again used \cite[Proposition I.4.49 (a)]{JS00}. Analogously, we get $[Y,Z\p p]\p p=[Y\p p,Z\p p]$ and the proof of (ii) is complete.
We now show (iii). First, we compute
\begin{equation}\label{eq:sbYsZs}
[\ol Y,\ol Z]=[Y,Z]-[Y\p p,Z]-[Y, Z\p p]+[Y\p p,Z\p p].
\end{equation}
By \eqref{eq:sbYsZs} and (ii),
since  $[Y\p p,Z\p p]-[Y, Z\p p]\in\Mloc$, we get $[\ol Y,\ol Z]\in\Mloc$ if and only if $[Y,Z]-[Y\p p,Z]\in\Mloc$. But this is the case if and only if $[Y,Z]$ and $[Y\p p,Z]$ are associated processes thus, by (ii), if and only if $[Y,Z]\p p=[Y\p p,Z\p p]$ holds. This shows (iii).
We now verify (iv). It is enough to show that if $[\ol Y,\ol Z]\in\Mloc$, then $[\ol Y,\ol Z]=0$. Therefore, let  $[\ol Y,\ol Z]\in\Mloc$ and $\Delta Y\Delta Z=0$. By (iii), we get $[Y\p p,Z\p p]=0$, since $[Y,Z]=0$ by assumption. Hence, by (ii), we obtain $[Y,Z\p p],[Y\p p,Z]\in\Mloc$ implying that $[Y,Z\p p]=[Y\p p,Z]=0$ since these are increasing processes starting at zero. Then (iv) follows immediately from \eqref{eq:sbYsZs}.
To verify (v), we observe that, since $[Y,Z]\p p=0$ by assumption, (iii) yields $[\ol Y,\ol Z]\in\Mloc$ if and only if $[Y\p p,Z\p p]=0$ or, equivalently, if and only if $\Delta Y\p p\Delta Z\p p=0$. Finally we show (vi). By the property of the predictable brackets of a locally square integrable local martingale, we have $\aPP{\ol Y}{\ol Y}=[\ol Y,\ol Y]\p p$. By \eqref{eq:sbYsZs} with $Z=Y$ and (ii), by the properties of the square brackets and of the dual predictable projection, we get
\[[\ol Y,\ol Y]\p p=([Y,Y]-[Y\p p,Y])\p p=(Y-\Delta Y\p p\cdot Y)\p p= Y\p p-\Delta Y\p p\cdot Y\p p=(1-\Delta Y\p p)\cdot Y\p p.\] The proof of the lemma is now complete.
\end{proof}

Notice that the relation $[Y,Z]\p p=[Y\p p,Z\p p]$ of Lemma \ref{lem:orth.mart} (iii) does not hold, in general: For example if $Y$ is a standard Poisson process and $Z=Y$, then $[Y,Y]_t\p p=Y\p p_t= t$, while $[Y\p p,Y\p p]=0$.

We also notice that the condition $\Delta Y\Delta Z=0$ is not sufficient for $\Delta Y\p p\Delta Z\p p=0$, in general, as the  following counterexample shows.

\begin{cexample}\label{cex:jum.com}\footnote{This counterexample has been suggested to us by Yuliya Mishura and Alexander Gushchin.}
Let $\Abb=(\Ascr_t)_{t\geq0}$ be such that $\Ascr_t$ is trivial for $t<1$ and $\Ascr_t=\Fscr$ for $t\geq1$. Then $\Abb$ is obviously a right-continuous filtration. Let $A,B\in\Fscr$ be disjoint. Then, the processes $Y$ and $Z$ defined by $Y_t=1_A1_{\{t\geq1\}}$ and $Z_t=1_B1_{\{t\geq1\}}$, respectively, are point processes with respect to $\Fbb$ and they satisfy $\Delta Y\Delta Z=0$. For the dual predictable projections $Y\p p$ and $Z\p p$ we have $Y\p p_t=\Pbb[A]1_{\{t\geq1\}}$ and $Z_t\p p=\Pbb[B]1_{\{t\geq1\}}$, respectively, and they have a common jump in $t=1$.
\end{cexample}

\paragraph*{Acknowledgement.} PDT gratefully acknowledges Martin Keller-Ressel and funding from the German Research Foundation (DFG) under grant ZUK 64.  MJ was supported by  the ``Chaire March\'es en Mutation'',
 French Banking Federation and ILB, Labex ANR 11-LABX-0019.

\end{document}